\documentclass[reqno]{amsart}

\usepackage[a4paper, left=1.5in,right=1.5in,top=1.25in]{geometry}

\usepackage{mathrsfs}
\usepackage{amsfonts}
\usepackage{amsmath}
\usepackage{amssymb}
\usepackage[table,xcdraw]{xcolor}
\usepackage{pifont}
\usepackage{mathtools}
\usepackage{float}
\usepackage{xcolor}
\usepackage{mathdots}
\usepackage{enumitem}
\usepackage{array}
\usepackage{subcaption}
\usepackage{mathdots}
\usepackage{tikz}
\usepackage{tikz-cd}
\usetikzlibrary{arrows,automata}
\usetikzlibrary{decorations.markings}
\usetikzlibrary{automata, positioning, arrows}
\usetikzlibrary{math}
\usepackage{hyperref}
\usepackage{centernot}

\newcommand{\pres}[3]{\textnormal{#1} \langle #2 \mid #3 \rangle}

\newcommand{\Z}{\mathbb{Z}}

\newcommand{\cH}{\mathcal{H}}

\newcommand{\cM}{\mathcal{M}}

\newtheorem{innercustomgeneric}{\customgenericname}
\providecommand{\customgenericname}{}

\newtheorem{theorem}{Theorem} 
\numberwithin{theorem}{section}
\newtheorem*{theorem*}{Theorem} 
\newtheorem{lemma}[theorem]{Lemma}     
\newtheorem{corollary}[theorem]{Corollary}

\newtheorem*{mainlemma*}{Main Lemma}

\newtheorem*{conjecture*}{Conjecture}

\theoremstyle{definition}

\newtheorem*{question*}{Question}

\numberwithin{example}{section}

\newtheorem*{remark*}{Remark}

\DeclareMathOperator{\SL}{SL}

\newcommand{\N}{{\mathbb{N}}}
\newcommand{\SLP}[1]{\SL_2(\mathbb{Z}[\frac{1}{#1}])}

\begin{document}

\title[The abelianization of $\SLP{m}$]{The abelianization of $\SLP{m}$}

\author{Carl-Fredrik Nyberg-Brodda}
\address{School of Mathematics, Korea Institute for Advanced Study (KIAS), Seoul 02455, Republic of Korea}
\email{cfnb@kias.re.kr}
\thanks{The author is supported by Mid-Career Researcher Program (RS-2023-00278510) through the National Research Foundation funded by the government of Korea.}

\subjclass[2020]{13D03 (primary), 20H25, 20F05 (secondary)}

\date{\today}


\keywords{}

\begin{abstract}
For all $m \geq 1$, we prove that the abelianization of $\SLP{m}$ is (1) trivial if $6 \mid m$; (2) $\Z / 3\Z$ if $2 \mid m$ and $\gcd(3,m)=1$; (3) $\Z / 4 \Z$ if $3 \mid m$ and $\gcd(2,m)=1$; and (4) $\Z / {12}\Z \cong \Z / 3\Z \times \Z / 4\Z$ if $\gcd(6,m)=1$. This completes known computational results of Bui Anh \& Ellis for $m \leq 50$. The proof is completely elementary, and in particular does not use the congruence subgroup property. We also find a new presentation for $\SLP{2}$. This presentation has two generators and three relators. Thus, $\SLP{2}$ admits a presentation with deficiency equal to the rank of its Schur multiplier. This also gives new and very simple presentations for the finite groups $\SL_2(\Z / m \Z)$, where $m$ is odd.
\end{abstract}

\maketitle

Throughout this article, $\SLP{m}$ denotes the multiplicative group of $2 \times 2$ matrices over the ring $\Z[\frac{1}{m}]$ with determinant $1$. In particular, $\SLP{m} = \SLP{n}$ if and only if $m$ and $n$ have the same set of prime factors.  We will use these groups to yield a formula for the abelianization of $\SLP{m}$ for all $m$. Computing the integral homology groups $H_k(\SLP{m}, \Z)$ is a very difficult problem in general; Bui Anh \& Ellis \cite{BuiAnh2014} did so for small $k$ and $m \leq 50$. The result contained in this present note thus completely solves the case of $k=1$. 

\begin{remark*}
In the final days of preparing this note, a preprint by Mirzaii \& Torres~P\'{e}rez \cite{MirzaiiPerez} appeared on the arXiv, containing numerous interesting results on $\operatorname{E}_2(A)$ for arbitrary rings. A particular case of their methods computes the abelianization of $\SLP{m}$ (their Proposition~4.3) as in this note. However, their method is quite different from and less elementary than that contained in this note, and in particular their note does not contain a presentation for $\SLP{2}$. 
\end{remark*}

\vspace{-0.1cm}

\section{}

For $m \geq 1$, let us define a three-relator group 
\begin{equation}
\cH_m = \pres{}{x,y}{x^m y x^m = yx^my, \: y^mxy^m = xy^mx, \: (x^2y^m)^4=1}.
\end{equation}
For $m=1$ we find a well-known presentation for $\SL_2(\Z)$, i.e.\ $\cH_1 \cong \SL_2(\Z) = \SLP{1}$. For $m \geq 1$, let $\varphi_m \colon \cH_m \to \SLP{m}$ be defined by 
\begin{equation}\label{Eq:varphi-map-def}
\varphi_m(x) = A = \begin{pmatrix}
1 & 0 \\ 1 & 1 
\end{pmatrix}, \quad \varphi_m(y) = Q_m = \begin{pmatrix}
1 & -1/m \\ 0 & 1 
\end{pmatrix}.
\end{equation}

This map shall be used to prove all the results of this article.

\begin{lemma}\label{Lem:Exists-surjection}
For all $m \geq 1$, $\varphi_m \colon \cH_m \to \SLP{m}$ is a surjective homomorphism.
\end{lemma}
\begin{proof}
Let $A = \varphi_m(x)$ and $Q_m = \varphi_m(y)$, as in \eqref{Eq:varphi-map-def}. To prove that $\varphi_m$ is a homomorphism, we check that all relations in $\cH_m$ are satisfied by $A$ and $Q_m$, viz.\
\begin{align*}
A^m Q_m A^m &= \begin{pmatrix}
0 & -1/m \\ m & 0
\end{pmatrix} = Q_m A^m Q_m, \\
Q_m^m A Q_m^m &= \begin{pmatrix}
0 & -1 \\ 1 & 0
\end{pmatrix} = A Q_m^m A, 
\end{align*}
and clearly $A^2 Q_m^m = \begin{psmallmatrix}1 & -1 \\ 2 & -1 \end{psmallmatrix}$, an element of order $4$. Thus $\varphi_m$ is a homomorphism. To see that $\varphi_m$ is surjective, it suffices to note that $A, Q_m$ generate $\SLP{m}$. This is not difficult; indeed, it is well-known that $\SLP{m}$ is generated by the three matrices
\[
A = \begin{pmatrix}
1 & 0 \\ 1 & 1 
\end{pmatrix}, \quad B = \begin{pmatrix}
0 & 1 \\ -1 & 0
\end{pmatrix}, \quad \text{and} \quad U_m = \begin{pmatrix}
m & 0 \\ 0 & 1/m 
\end{pmatrix},
\]
(for the short proof of this fact, see e.g.\ \cite{Mennicke1967}). But $B = A^{-1} Q_{m}^{-m} A^{-1}$ and $U_m = B^{-1} Q_m^{-1} A^{-m}Q_m^{-1}$. Thus, $\varphi_m$ is surjective. 
\end{proof}

\begin{theorem}\label{Thm:Main-thm}
For all $m \geq 1$, we have 
\[
\SLP{m}^{\operatorname{ab}}\cong \cH_m^{\operatorname{ab}} \cong \begin{cases*} 
1 & if $6 \mid m$, \\
\Z / 3 \Z & if $2 \mid m$ and $\gcd(3,m)=1$,\\
\Z / 4 \Z & if $3 \mid m$ and $\gcd(2,m)=1$, \\
\Z / 12\Z & if $\gcd(m,6)=1$.
\end{cases*}
\]
\end{theorem}
\begin{proof}
Deriving a formula for the abelianization of $\cH_m$ is a simple exercise in linear algebra. Indeed, by setting up the relation matrix for $\cH_m$ and computing subdeterminants, one finds that $H_1(\cH_m, \Z) \cong \Z / \gcd(m^2 + 1, 12m, 4m^2 + 8) \Z$ and this is easily seen to be equivalent to the specified condition.

By Lemma~\ref{Lem:Exists-surjection}, it follows that $\varphi_m$ induces a surjection $\varphi_m^{\operatorname{ab}} \colon \cH_m^{\operatorname{ab}} \to \SLP{m}^{\operatorname{ab}}$, and hence the latter group is also finite and indeed a quotient of the former. Let $S_m = \SLP{m}^{\operatorname{ab}}$. If $6 \mid m$, then $\cH_m^{\operatorname{ab}}$ and hence also $S_m$ is trivial. If $2 \mid m$ and $\gcd(3,m)=1$, then $|S_m| \leq 3$, but in this case $\SLP{m}$ surjects $\SL_2(3)$ which has abelianization $\Z / 3\Z$. Hence $S_m = \Z / 3\Z$ in this case. If $3 \mid m$ and $\gcd(2,m)=1$, then $|S_m| \leq 4$. Since $\SLP{m}$ surjects $\SL_2(\Z / 4\Z)$, and since this group has abelianization $\Z / 4\Z$, it follows that $S_m = \Z / 4\Z$ in this case. Finally, if $\gcd(m,6)=1$, then $|S_m| \leq 12$ and combining the above arguments shows that $S_m \cong \Z / 4\Z \times \Z / 3\Z \cong \Z / 12\Z$. 
\end{proof}

As a particular case we find ($m=1$) that $\SL_2(\Z)^{\operatorname{ab}} \cong \Z / 12\Z$, which is of course well-known. The results for $m=2,3$ are well-known, and for prime $p>3$ it is classical and follows from Serre's amalgam decomposition \cite[\S1.4]{Serre1980} that the abelianization of $\SLP{p}$ is $\Z / 12\Z$ (see also \cite{Adem1998}). Bui Anh \& Ellis \cite[Table~1]{BuiAnh2014} computed the abelianization (and many other homology groups) of $\SLP{m}$ for $m \leq 50$, and their computations agree with Theorem~\ref{Thm:Main-thm}.

\section{}

Above, we proved that $\varphi_m$ is a surjection for all $m$. A simple argument based on the determination of $H_2(\SLP{p}, \Z)$ by Adem \& Naffah \cite{Adem1998} shows that, for sufficiently large primes $p$, we must have $\SLP{p} \not\cong \cH_p$. K.\ Hutchinson (private communication) has shown the author a beautiful argument from algebraic $K$-theory which shows that $\varphi_m$ is not an isomorphism for $m \geq 3$. However, we shall now prove that $\varphi_2$ is an isomorphism. This yields a new presentation for $\SLP{2}$, simpler than that given by Serre \cite{Serre1980} and Behr \& Mennicke \cite{Behr1968}. 

\begin{theorem}
The map $\varphi_2$ is an isomorphism. In particular,
\[
\SLP{2} \cong \pres{}{x,y}{x^2 y x^2 = yx^2y, \: y^2xy^2 = xy^2x, \: (x^2y^2)^4=1}.
\]
\end{theorem}
\begin{proof}
The following presentation for $\SLP{2}$ can be found in Serre \cite[\S1.4]{Serre1980} and in Behr \& Mennicke \cite{Behr1968}: 
\begin{equation}\label{Eq:serre-pres}
\SLP{2} \cong \pres{}{a,b,u}{b^4 = 1, b^2 = (bu)^2=(ba)^3=(bua^2)^3, u^{-1}au = a^4},
\end{equation}
where the isomorphism is given by identifying $a, b, u$ with $A,B,U_2$ as defined above. Thus, it suffices to prove that the group presented by the presentation in \eqref{Eq:serre-pres} is isomorphic to $\cH_2$ under $\varphi_2$.  

Adding the new generator $q = bua^2u^{-1}b^{-1}$ (as $Q_2 = BU_2A^2U_2^{-1}B^{-1}$) to the presentation \eqref{Eq:serre-pres}, we have the equalities $b = a^{-1}q^{-2}a^{-1}$ and $u = b^{-1}q^{-1}a^{-2}q^{-1}$, as in the proof of Lemma~\ref{Lem:Exists-surjection}. We can thus eliminate these two generators $b$ and $u$, simplifying the presentation to one with the two generators $a$ and $q$. The relators of this presentation, we claim, can be taken as the relators of $\cH_2$ (with the relabelling $x=a$ and $y=q$). The relation $q = bua^2u^{-1}b^{-1}$ becomes, after simplification,
\begin{equation}\label{Eq:am-q-am}
a^2qa^2 = qa^2q
\end{equation}
which is the first relation of $\cH_2$. Next, we observe that $bu = q^{-1}a^{-2}q^{-1}$, and hence the relation $(ab)^3 = b^2$ becomes $(q^{-2}a^{-1})^3 = (a^{-1}q^{-2}a^{-1})^2$, which simplifies to 
\begin{equation}\label{Eq:qm-a-qm}
q^2a q^2 = aq^2a,
\end{equation}
the second relation of $\cH_2$. Third, the relation $b^4 = 1$ becomes rewritten to $(a^{-1}q^{-2}a^{-1})^4 = 1$, which is easily seen to be equivalent to 
\begin{equation}\label{Eq:a2qm4}
(a^2q^2)^4 = 1,
\end{equation}
being the third (and final) relator of $\cH_2$. We now must simply check that the remaining relators $(bu)^2=b^2, \: (bua^2)^3=b^2$, and $u^{-1}au = a^4$, when rewritten over $a, q$, are redundant modulo \eqref{Eq:am-q-am}, \eqref{Eq:qm-a-qm}, and \eqref{Eq:a2qm4}. This is a routine check, and can be done even without use the relator \eqref{Eq:a2qm4}.
\end{proof}

It would be interesting to understand $\ker \varphi_m$ for $m \geq 3$, and in particular to know how many (necessarily finitely many) relations, and which, must be added to $\cH_m$ to make it isomorphic to $\SLP{m}$. 

For $r \in \N$ with $\gcd(r,p)=1$, Mennicke's theorem \cite{Mennicke1967} asserts that the principal congruence subgroup of level $r$ coincides with the normal closure of $A^r$ in $\SLP{p}$. Hence, we also find new presentations for the quasisimple groups $\SL(2,r)$ when $r$ is an odd prime, and indeed more generally:

\begin{corollary}
For all odd $r$ we have 
\[
\SL_2(\Z / r\Z) \cong \pres{}{x,y}{x^2 y x^2 = yx^2y, \: y^2xy^2 = xy^2x, \: (x^2y^2)^4 = 1, \: x^r = 1}.
\]
\end{corollary}

Note that deficiency zero presentations of $\SL(2,\Z / r\Z)$ (which has trivial Schur multiplier) for odd $r$ are known (see \cite{Campbell1980} and the remark following Theorem~4 therein). The above presentation for the group has the advantage of being very simple to remember.

\section*{Acknowledgements}

I would like to thank G. Ellis, K. Hutchinson, and A. Rahm for helpful discussions, pointers to the literature, and encouragement to write up this short note.

\bibliographystyle{amsalpha}
\bibliography{sl2z1m-presentation.bib}

\providecommand{\bysame}{\leavevmode\hbox to3em{\hrulefill}\thinspace}
\providecommand{\MR}{\relax\ifhmode\unskip\space\fi MR }
\providecommand{\MRhref}[2]{%
  \href{http://www.ams.org/mathscinet-getitem?mr=#1}{#2}
}
\providecommand{\href}[2]{#2}
\begin{thebibliography}{Men67}

\bibitem[AN98]{Adem1998}
Alejandro Adem and Nadim Naffah, \emph{On the cohomology of {${\rm SL}_2(\bold
  Z[1/p])$}}, Geometry and cohomology in group theory ({D}urham, 1994), London
  Math. Soc. Lecture Note Ser., vol. 252, Cambridge Univ. Press, Cambridge,
  1998, pp.~1--9.

\bibitem[BM68]{Behr1968}
H.~Behr and J.~Mennicke, \emph{A presentation of the groups {${\rm
  PSL}(2,\,p)$}}, Canadian J. Math. \textbf{20} (1968), 1432--1438.

\bibitem[CR80]{Campbell1980}
C.~M. Campbell and E.~F. Robertson, \emph{A deficiency zero presentation for
  {${\rm SL}(2,\,p)$}}, Bull. London Math. Soc. \textbf{12} (1980), no.~1,
  17--20.

\bibitem[Men67]{Mennicke1967}
J.~Mennicke, \emph{On {I}hara's modular group}, Invent. Math. \textbf{4}
  (1967), 202--228.

\bibitem[MP24]{MirzaiiPerez}
B.~Mirzaii and E.~Torres P\'{e}rez, \emph{The abelianization of the elementary
  group of rank two}, 2024, arXiv:2401.06330.

\bibitem[Ser80]{Serre1980}
Jean-Pierre Serre, \emph{Trees}, Springer-Verlag, Berlin-New York, 1980,
  Translated from the French by John Stillwell.

\bibitem[TE14]{BuiAnh2014}
Bui~Anh Tuan and Graham Ellis, \emph{The homology of {$SL_2(\Bbb Z[1/m])$} for
  small {$m$}}, J. Algebra \textbf{408} (2014), 102--108.

\end{thebibliography}

 \end{document}